\documentclass[12pt,oneside]{amsart} 
\usepackage{amssymb,amscd}
\usepackage{euscript}

      \theoremstyle{plain}
      \newtheorem{theorem}{Theorem}[section]
      \newtheorem{lemma}[theorem]{Lemma}
      \newtheorem{corollary}[theorem]{Corollary}
      \newtheorem{proposition}[theorem]{Proposition}

      \newtheorem{definition}[theorem]{Definition}     
           \newtheorem{assumption}[theorem]{Assumption}     
          
\numberwithin{equation}{section}

      \makeatletter
      \def\@setcopyright{}
      \def\serieslogo@{}
      \makeatother

\def\A{\EuScript{A}} 
\def\B{\EuScript{B}} 
\def\E{\mathcal{E}}
\def\L{\mathcal{L}}
\def\F{\mathcal F}

\def\H{{H}} %\def\H{\mathcal{H}}
\def\h{\mathcal{H}}
\def\M{\mathcal{M}}

\def\n{\mathcal N}

\def\c{{C}}%\def\c{\EuScript{C}}
\def\R{\mathbb R}
\def\rd{{\mathbb R ^d}}

\def\Z{\mathbb Z}
\def\N{\mathbb N}
\def\T{\mathcal T}

\def\dist{\text{dist}}

\def\Id{\text{Id}}
\def\e{\epsilon}

\def\bv{\mathbf v}

\def\I{{I}}

\def\p{\varphi}
\def\P{\Phi}
\def\q{\eta}

\def\loc{\text{loc}}
\def\EE{\mathcal{E}}
\def\vv{{V}}
\def\U{{\tilde{U}}}
\def\UU{{\tilde{\mathcal{U}}}}
\def\V{{U}}
\def\VV{\mathcal{U}}
\def\P{{\Phi}}
\def\PP{\mathcal{P}}
\def\QED{\hfill\hfill{\square}}

\begin{document}

\date{\today}
\author{Boris Kalinin$^\ast$ and Victoria Sadovskaya$^{\ast\ast}$}

\address{Department of Mathematics, The Pennsylvania State University, University Park, PA 16802, USA.}
\email{kalinin@psu.edu, sadovskaya@psu.edu}

\title [On regularity of conjugacy between linear cocycles]
{On regularity of conjugacy between linear cocycles over partially hyperbolic systems} 

\thanks{$^{\ast}$  Supported in part by Simons Foundation grant 855238}
\thanks{$^{\ast\ast}$ Supported in part by Simons Foundation grant 00002874}

%%%%%%%%%%%%%%%%%%%%%%%%%%%%%%%

\begin{abstract}
We consider H\"older continuous $GL(d,\R)$-valued  cocycles, and more generally linear cocycles, over an accessible volume-preserving center-bunched partially hyperbolic diffeomorphism. We study the regularity of a conjugacy between two cocycles. We establish continuity of a measurable conjugacy between {\em any} constant $GL(d,\R)$-valued cocycle and its perturbation. We deduce this from our main technical result on continuity of  a measurable conjugacy between a fiber bunched linear cocycle and a cocycle with a certain block-triangular structure. The latter class covers constant cocycles with one  Lyapunov exponent. We also establish a result of independent interest on continuity of measurable solutions for twisted vector-valued cohomological equations over partially hyperbolic systems. In addition, we give more general versions of earlier results on regularity of invariant subbudles, Riemannian metrics, and conformal structures.
\end{abstract}

\maketitle

%%%%%%%%%%%%%%%%%%%%%%%%%%%%
%%%%%%%%%% Introduction %%%%%%%%%%%
%%%%%%%%%%%%%%%%%%%%%%%%%%%%

\section{Introduction and main results}

Cocycles and their cohomology  play an important role in dynamics.
%For example, they appear in the study of time changes  for flows and group actions, existence and smoothness of absolutely continuous invariant measures, existence and smoothness of conjugacies between dynamical systems, rigidity in dynamical systems and  group actions.
In this paper we consider $GL(d,\R)$-valued  cocycles, and more generally linear cocycles, 
over a volume-preserving partially hyperbolic diffeomorphism $f$ of a compact manifold $\M$.
The prime examples are given by the differential of $f$ and its restrictions  to invariant 
subbundles, for example stable, unstable, or center. Such cocycles are used in the study of
dynamics and rigidity of hyperbolic and partially hyperbolic systems. 

First we discuss $GL(d,\R)$-valued  cocycles.
 
  \begin{definition} 
  Let $A:\M\to GL(d,\R)$ be a continuous function.
The $GL(d,\R)$-valued cocycle over $f$ generated by  $A$ is the map 
 $\A:\,\M \times \Z \,\to GL(d,\R)$ defined  as follows: for $x\in \M$ and $n\in \N$,  
\vskip.1cm \noindent 
$$\A_x^0=Id,\,  \quad \A_x^n = A(f^{n-1} x)\circ \cdots \circ A(x) \quad \text{and }\quad 
\A_x^{-n}= (\A_{f^{-n} x}^n)^{-1}.
$$
\end{definition}

 If the tangent bundle of $\M$ is trivial, $T\M= \M\times \R^d$, 
then the differential $Df$ can be viewed as a $GL(d,\R)$-valued cocycle with
$  A(x)=Df_x$ and  $\A_x^n=Df^n_x.$ 
\vskip.1cm
A natural  equivalence relation for cocycles is defined as follows.  

\begin{definition} \label{cohomology def}
$\,\,GL(d,\R)$-valued cocycles\, $\A$\, and\, $\B$\, over $f$ are\, (measurably\, or \\continuously) {\em cohomologous} 
if there exists a (measurable or continuous) function \\ $C:\M\to GL(d,\R)$ such that
\begin{equation} \label{C def}
\B_x=C(f x) \circ \A_x \circ  C(x)^{-1} 
  \;\text{ (almost everywhere  or for all  $x\in \M$).}
\end{equation}
We refer to $C$ as a (measurable or continuous) {\em conjugacy} between $\A$  and $\B$. %It is also called a {\em transfer map}. 
%$$ \,\B_x^n=C(f^n x) \circ \A_x^n \circ  C(x)^{-1}   \;\text{ for all }n\in \Z \text{ and }x\in \M,  $$ 
%equivalently, for the generators  $B(x)=C(fx) \circ A(x) \circ C(x)^{-1}\;\text{ for all }x\in \M.$
\end{definition}

%For the differential example above, $C(x)$ can be viewed as a coordinate change on $T_x\M$.
%In the context of cocycles over partially hyperbolic systems, the most studied cohomology problem is finding sufficient conditions for existence of a continuous (or more regular) conjugacy.

We consider the question whether a measurable conjugacy between two cocycles is continuous. A positive answer was obtained by Wilkinson  in \cite{W} for H\"older continuous
 $\R$-valued cocycles over an accessible center-bunched volume preserving 
 $C^2$  partially hyperbolic  diffeomorphism. 

For cocycles with values in non-commutative groups, studying cohomology is more difficult. 
Usually additional assumptions related to their growth are made, such as fiber bunching. 
The latter means that non-conformality of the cocycle is dominated by the expansion and 
contraction in the base, see Definition \ref{FB}.
The first result on continuity of a measurable conjugacy for non-commutative cocycles over 
partially hyperbolic systems was obtained in \cite[Theorem 4.2]{KS15}. It extended earlier 
results for cocycles over hyperbolic diffeomorphisms \cite{Sch,NP,PW,S15}.
There we established  continuity of a measurable conjugacy between H\"older continuous fiber 
bunched cocycles, one of which is uniformly quasiconformal. A cocycle $\A$ is  {\em uniformly 
quasiconformal}  if $\,\| \A^n_x\|\cdot \|( \A^n_x)^{-1}\|$ is uniformly bounded in $x\in \M$ and $n\in \Z$.     

In contrast to scalar cocycles, a measurable conjugacy between $GL(d,\R)$-valued cocycles
 is not always continuous, even when $f$ is hyperbolic and both cocycles are close to the identity.
   Indeed, in \cite{PW} Pollicott and Walkden constructed  smooth 
   $GL(2,\R)$-valued cocycles over an Anosov toral automorphism of the form 
\begin{equation}\label{PWex}
\A_x=\left[ \begin{array}{cc} a(x) & b(x) \\ 0 & 1 \end{array} \right] 
\quad\text{and} \quad
\B_x= \left[ \begin{array}{cc} a(x) & 0 \\ 0 & 1 \end{array} \right]
\end{equation}
that are measurably (with respect to the Lebesgue measure $m$), but not continuously cohomologous.
We note that these cocycles have two Lyapunov exponents, $0$ and $\int \log a(x) \, dm<0$.
Thus in general one can not expect continuity of a measurable conjugacy in case of more than one Lyapunov  exponent.

The next theorem gives a positive result for a constant cocycle $\A$
with one Lyapunov  exponent, which means that all eigenvalues 
of the matrix $A$ generating $\A$ have the same modulus.

\begin{assumption} \label{PH}
In this paper, $f$ is an accessible center-bunched $C^2$ partially hyperbolic  
diffeomorphism of a compact manifold $\M$ preserving a  volume $\mu$. 
(See Section \ref{partial} for details.)
\end{assumption}

\begin{theorem} \label{constant 1exp}
Let $\A$ be a constant $GL(d,\R)$-valued cocycle with one Lyapunov exponent and let $\B$ be an su-$\beta$-H\"older fiber bunched $GL(d,\R)$-valued cocycle over $f$.  
Then any $\mu$-measurable conjugacy between $\A$ and $\B$ coincides $\mu$-a.e. with an su-$\beta$-H\"older conjugacy.
%conjugacy $C$ which intertwines their holonomies.
\end{theorem}

We say that a function is su-$\beta$-H\"older if it is continuous on $\M$ and $\beta$-H\"older 
continuous along the leaves of stable and unstable manifolds for $f$, see Section \ref{su func}.

As a corollary of Theorem \ref{constant 1exp}, we obtain continuity of a measurable conjugacy 
between {\em any} constant $GL(d,\R)$-valued cocycle and its perturbation, without 
fiber bunching or one Lyapunov exponent assumptions on either cocycle.

\begin{theorem} \label{constant perturb}
Let  $\A$ is be a constant $GL(d,\R)$-valued cocycle over $f$. Then
for any H\"older continuous $GL(d,\R)$-valued cocycle $\B$ sufficiently $C^0$ close to $\A$, 
any $\mu$-measurable conjugacy between $\A$ and $\B$ coincides  $\mu$-a.e. with an su-H\"older 
conjugacy.
\end{theorem}

We deduce Theorem \ref{constant 1exp} from a more general result, Theorem \ref{main} below. 
It holds  in a broader context of linear cocycles on vector bundles, 
 see Section \ref{lin coc} for details. Also, instead of a constant cocycle with one exponent  
 we consider  a cocycle 
 with a certain ``block-triangular" structure. As we show in Proposition \ref{prop distortion}, this structure  implies 
 that the cocycle is fiber bunched and has  one Lyapunov exponent for each  $f$-invariant ergodic measure.  For a hyperbolic $f$, the converse also holds by \cite[Theorem 3.9]{KS13}. 
However, the converse is not known and may not hold in general in the  partially hyperbolic case,
where existing results, such as \cite[Theorem 3.4]{KS13}, give a weaker structure.

 We say that a linear cocycle $\tilde \A$ on a vector bundle $\tilde \E$ over $\M$ is {\em uniformly bounded} 
 if $\|\tilde \A^n_x\|$ is uniformly bounded in $x\in \M$ and $n\in \Z$. This notion does not depend 
 on the choice of a continuous norm on $\tilde \E$.  
 
\begin{theorem} \label{main}
Let $\E$ and $\E'$ be  $\beta$-H\"older vector bundles over $\M$, or more generally 
su-$\beta$-H\"older subbundles of $\beta$-H\"older vector bundles over $\M$.

Let $\A$ be an su-$\beta$-H\"older linear cocycle on $\E$ over $f$. Suppose that there exist 
a flag of su-$\beta$-H\"older  $\A$-invariant sub-bundles 
\begin{equation} \label{flagH}
\{0\} =\vv^0 \subset \vv^1 \subset \dots  \subset \vv^{k-1} \subset \vv^k = \E
\end{equation}
and a positive su-$\beta$-H\"older function $\psi : \M \to \R$ so that the quotient-cocycles 
induced by the cocycle $\psi \A$ on $\vv^{i}/\vv^{i-1}$ are uniformly bounded for $i=1, ... , k$. 

Let $\B$ be an su-$\beta$-H\"older fiber bunched cocycle over $f$ on $\E'$. Then any $\mu$-measurable 
conjugacy between $\A$ and $\B$ coincides $\mu$-a.e. with an su-$\beta$-H\"older 
conjugacy which intertwines their holonomies (see  Definition \ref{def intertwine}).
\end{theorem}

This theorem extends both the partially hyperbolic result \cite[Theorem 4.2]{KS15} for uniformly quasiconformal $\A$ and the hyperbolic result \cite[Theorem 2.1]{KSW}. 

In the hyperbolic case, DeWitt recently showed in \cite{DW} that fiber bunching of $\B$ 
 can be verified if  $\B$ is measurably conjugate to a cocycle $\A$ taking values in a Zimmer block. 
 This assumption on $\A$ 
 is weaker than in Theorem \ref{constant 1exp} and stronger than in Theorem \ref{main}.
 This result strongly relies on hyperbolicity and periodic points.

One of the difficulties in the partially hyperbolic case compared to the hyperbolic one is
obtaining global regularity of conjugacies or invariant objects from (essential) 
regularity along the stable and unstable foliations. 
This step is simple for H\"older regularity in the  hyperbolic case due to the local product structure of 
the stable and unstable foliations. To obtain continuity in the partially hyperbolic case we use results by Avila, Santamaria, and Viana \cite{ASV} for accessible center bunched volume preserving $f$.
For  {\em scalar} cocycles, global H\"older continuity of the conjugacy (with reduced H\"older  exponent) was establishes by Wilkinson \cite{W}. However, accessibility is not known to yield global H\"older continuity of conjugacies or invariant objects for $GL(d,\R)$-valued  cocycles. This creates a mismatch between H\"older input and continuous output of the results,
and hence difficulties in using them repeatedly or inductively, as continuity is not enough to work with.
We overcome these difficulties  by using holonomies and by obtaining the results with 
su-$\beta$-H\"older regularity for both input and output. 
%(which reduces to $\beta$-H\"older in hyperbolic case and yields sharp conclusions there)
We also give more general versions for various earlier results under the assumptions of
su-$\beta$-H\"older regularity or existence of holonomies.

We also establish a result 
of independent interest on continuity of measurable solutions for twisted vector-valued cohomological
equations over partially hyperbolic systems, Theorem \ref{twist-meas}, which covers the usual
 (untwisted) scalar and vector-valued cocycles as particular cases. This result plays a key role in the proof of Theorem \ref{main}.

\vskip.1cm
The paper is structured as follows. We describe the setting and introduce the terminology in Section 2.
We prove Theorem \ref{main} in Section \ref{cocycle proofs}, and deduce  Theorems  \ref{constant 1exp} and  \ref{constant perturb} in Section \ref{proof of constant}. The results  for twisted cohomological 
equations are stated and proved in Section \ref{twisted}, and those on regularity of invariant
subbudles, Riemannian metrics, and conformal structures in Section \ref{invariant}.

%%%%%%%%%%%%%%%%%%%%%%%%%%%%

\section{Preliminaries}

 \subsection{Partially hyperbolic diffeomorphisms.} \label{partial}  $\;$
 
\noindent Let $\M$ be a compact connected smooth manifold.
 A diffeomorphism $f$ of $ \M$ is {\em partially hyperbolic} if
there exist a $Df$-invariant splitting of the tangent bundle 
$$T\M =E^s\oplus E^c \oplus E^u$$ with non-trivial $E^s$ and $E^u$, 
 and  a Riemannian 
metric on $\M$ for which one can choose continuous positive 
functions $\nu<1<\hat\nu\,$, $\gamma,$ $\hat\gamma\,$ such that 
for any $x \in \M$ and unit vectors  
$\,\bv^s\in E^s(x)$, $\,\bv^c\in E^c(x)$, and $\,\bv^u\in E^u(x)$
\begin{equation}\label{partial def}
\|Df_x(\bv^s)\| < \nu(x) <\gamma(x) <\|Df_x(\bv^c)\| < \hat\gamma(x) <
\hat\nu(x) <\|Df_x(\bv^u)\|.
\end{equation}
%We also choose continuous  functions $\mu$ and $\hat \mu$ such that for all $x$ in $\M$
%\begin{equation}\label{kappa} \mu(x)<\|Df_x(\bv^s)\|  \,\text{ if }\, \bv^x \in E^s(x) \quad\text{and}\quad \|Df_x(\bv^u)\| < \hat\mu (x)^{-1}\text{ if }\, \bv^u \in E^u(x). \end{equation}
 The sub-bundles $E^s$, $E^u$, and $E^c$ are called, respectively, stable, unstable, and center.
$E^s$ and $E^u$  are tangent to the stable and unstable foliations $W^s$ and $W^u$
 respectively.  We denote by $W^s_{\text{loc}}(x)$ the {\em local stable manifold}, which is
  the ball in $W^s(x)$  centered at $x$ of a sufficiently small fixed radius, in the distance  $\dist_{W^s}$
along the leaf.
 
  An {\it $su$-path} in $\M$ is a concatenation 
 of finitely many subpaths which lie entirely in a single 
 leaf of $W^s$ or  $W^u$. A partially hyperbolic  
 diffeomorphism $f$  is called {\em accessible}  if any two points 
 in $\M$ can be connected by an $su$-path.

 We say that $f$ is {\em volume-preserving} if it has an invariant probability 
measure  $\mu$ in the measure class of a volume induced by a 
Riemannian metric. 
%It is conjectured that any essentially accessible $f$ is ergodic with respect to such  $\m$. The conjecture was proved in cite \cite{BW} under the assumption that $f$ is $C^2$ and center bunched.
%or that $f$ is $C^{1+\epsilon}$, $0<\epsilon<1$, and strongly center bunched.

The diffeomorphism $f$ is called {\em center bunched}\,
if the functions  $\nu, \hat\nu, \gamma, \hat\gamma$  can be 
chosen to satisfy
$\gamma^{-1} \hat \gamma<\nu^{-1}$ and
$\gamma^{-1} \hat \gamma<\hat\nu$. This implies that nonconformality 
of $Df|_{E^c}$ is dominated by contraction/expansion in $E^s$/$E^u.$ 

We recall that $f$ is {\em hyperbolic} if $E^c=\bf{0}$.  Hyperbolic diffeomorphisms are trivially center bunched, and accessible by the local product structure of stable and unstable manifolds. So our results apply to hyperbolic volume-preserving diffeomorphisms. 

%A $C^{1+\epsilon}$ diffeomorphism $f$ is called {\em strongly center bunched}\, if
% \begin{equation}\label{strong center bunching} \nu^\theta<\gamma \hat \gamma \quad \text{and} \quad \hat\nu^\theta<\gamma \hat \gamma \end{equation}
% for some $\theta \in (0, \epsilon)$ satisfying the inequalities $\,\nu\gamma^{-1}< \mu^\theta$ and $\,\hat\nu \hat\gamma^{-1}< \hat \mu^\theta$.These inequalities imply that $E^c$ is $\theta$-H\"older.
%Note that $(\gamma \hat \gamma)^{-1}$ is an estimate of non-conformality of $Df|_{E^c}$. 

 \subsection{H\"older continuous vector bundles}\label{bundle} $\;$
 
\noindent We consider  a $d$-dimensional $\beta$-H\"older, $0<\beta \le 1$,
  vector bundle  $P : \E \to \M $. This means that there exists
an open cover $\{ U_i \}_{i=1}^k$ of $\M$ with coordinate systems 
$$
\phi_i : P^{-1} (U_i) \to U_i \times \rd, \quad \phi_i (v) = (P(v),\Phi_i(v))
$$
such that every transition map $\phi_j \circ \phi_i^{-1}$ is a homeomorphism and its
 restriction to the fiber $\Phi_j \circ \Phi_i^{-1}|_{\{x\} \times \R^d}$ depends 
 $\beta$-H\"older continuously on $x$ as a linear map on $\R^d$. 
 %i.e. there is $C$  such that $\| L_x - L_y \| \le C \, \dist (x,y) ^\beta  $ for all $i,j$ and all $x,y \in U_i \cap U_j$. 
We can identify $\E$ with a $\beta$-H\"older sub-bundle of a trivial bundle
via 
$$\phi : \E \to \M  \times \R^{kd}\quad\text{with}\quad \phi (v) = (P(v), \rho_1 \Phi_1(v) \times ... \times \rho_k \Phi_k(v)),
$$
where $\{ \rho_i \}$ is a $\beta$-H\"older partition of unity for $\{ U_i \}$.

Using this embedding we equip $\E$ with the induced $\beta$-H\"older  
Riemannian metric, i.e., a family of inner products on the fibers, and
fix an identification $\I_{x,y} :\E_x \to \E_y$ of fibers at 
nearby points. We define the latter as 
$\Pi_y^{-1} \circ \Pi_x$, where $\Pi_x$ is the orthogonal projection in 
$\R^{kd}$ from $\E_x$ to the subspace which is the middle point of
the unique shortest geodesic between $\E_x$ and $\E_y$ in the 
Grassmannian of $d$-dimensional subspaces of $\R^{kd}$. 
The identifications $\{ \I_{x,y} \}$ satisfy $ \I_{x,y}= \I_{y,x}^{-1}$ and
vary $\beta$-H\"older continuously on a neighborhood of the diagonal in $\M \times \M$.
% and  satisfy for some constant $C$ and any unit vector $u \in \E_x \subset \R^{kd}$
%\begin{equation}\label{I} \I_{x,y}= \I_{y,x}^{-1}\quad \text{and }  \;\; \| \I_{x,y} u - u\| \le C \dist (x,y)^\beta ???
%\; \text{and hence } | \| \I_{x,y}\| -1| \le C \dist (x,y)^\beta.
%\end{equation}
%$$ \| \I_{xz}-\I_{yz} \circ \I_{xy}\|\le C \dist (x,y)^\beta$$

%%%%%%%%%%%%%%%%%%%%%%%%%%

 %%%%%%%%%%%%%%%%%%%%%%%%%
 
 \subsection{su-$\beta$-H\"older functions}\label{su func} $\;$ \\
We say that a function $\psi$ on $\M$ with values in a  metric space is {\em  s-$\beta$-H\"older}
if  $\psi$ is continuous on $\M$ and $\beta$-H\"older along the leaves of the stable foliation $W^s$, in the sense that there exists a constant $K$ such that 
$$d(\psi(x),\psi(y))\leq K\,\dist_{W^s}  (x,y)^{\beta} \;
\text{ for all $x\in \M$ and $y \in W^s_{\text{loc}}(x)$.}
$$
We define u-$\beta$-H\"older functions similarly and say a function is {\em su-$\beta$-H\"older} 
if it is s-$\beta$-H\"older and u-$\beta$-H\"older.

In the bundle setting, we similarly define the notion of an  su-$\beta$-H\"older subbundle 
$\E'$ of a $\beta$-H\"older vector bundle $\E$ by using identifications $\I_{x,y}$, 
%(i.e. $\I_{x,y}(\E_x')$ is H\"older close to $\E_y'$ in $\E_y$) 
 or equivalently using an embedding $\E'\subset \E \subset \M  \times \R^{kd}$ 
 and thus viewing $\E_x'$ as the Grassmannian-valued function.
  Using the embedding we can also define local identifications $\I_{x,y}'$ for $\E'$ 
  as we did for $\E$. They are continuous on a neighborhood of the diagonal in $\M \times \M$ 
  and $\beta$-H\"older along the leaves of $W^s$ and $W^u$ in the above sense. 
  Then for objects on an su-$\beta$-H\"older subbundle we define the notion of being su-$\beta$-H\"older 
using these identifications. %in the same way as for a $\beta$-H\"older bundle
In particular, using an embedding we can obtain an su-$\beta$-H\"older Riemannian metric on $\E'$.

 \subsection{Linear cocycles} $\;$ \label{lin coc}
 
 \noindent Let $f$ be a diffeomorphism of  $\M$ and let $P : \E \to \M$
be a $\beta$-H\"older vector bundle over $\M$.
A {\it linear cocycle}\, over $f$ is an automorphism of $\E$ that projects to $f$, that is, 
a homeomorphism $\A:\E\to \E$ such that  $\,P\circ \A= f \circ P$ and for each $x\in \M$ the map $\A_x:\E_x\to \E_{fx}$ between the fibers is a linear isomorphism. 
In the case  of a trivial vector bundle $\E=\M \times \rd$, any linear cocycle $\A$ can be identified 
with a $GL(d,\R)$-valued cocycle generated by the function $A(x)=\A_x \in GL(d,\R)$.

\vskip.1cm

We use the following notations for the iterates of $\A$: $\;\; \A_x^0=\Id,\,$ and for $n\in \N$,\,
$$
\begin{aligned}
&\A_x^n = \A_{f^{n-1} x}\circ \cdots \circ \A_{fx}\circ\A_x:\,\E_x\to \E_{f^nx} \quad\text{and}\quad\,\\
&\A_x^{-n}= (\A_{f^{-n} x}^n)^{-1}: \,\E_x\to \E_{f^{-n}x}.
\end{aligned}
$$
% Clearly, $\A$ satisfies the {\em cocycle equation}\, $\A^{n+k}_x= \A^n_{f^k x} \circ \A^k_x$.
\vskip.1cm

The prime examples of linear cocycles over $f$  are the differential $Df$ viewed as an 
automorphism of the tangent bundle $T\M$, and its restrictions to $Df$-invariant 
subbundles $\E'\subset T\M$ such as $E^s$, $E^u$, or $E^c$.  In these examples,
$$
  \A_x=\,D_xf \; \text{ and }\; \A^n_x\,=\,D_xf^n,   \;\quad\text{or }\quad\; \; \A_x=Df|_{\E'(x)} \text{ and }\; \A^n_x\,=Df^n|_{\E'(x)}.
  $$
Since these sub-bundles are H\"older continuous but usually not more regular, the H\"older category
is natural for applications.

\vskip.1cm
A linear cocycle $\A$ is called {\em $\beta$-H\"older} if $\A_x$ depends  $\beta$-H\"older continuously on $x$,
more precisely, if there exist a constant $c$ such that 
for all nearby points $x,y\in \M$
\begin{equation}\label{FHolder}
 \| \A_x - \I_{fx,fy}^{-1} \circ \A_y \circ \I_{x,y}\|  \le c\cdot  \dist (x,y)^{\beta'}
\end{equation}
where  $\| .\|$ is the operator norm. Similarly, we say that $\A$ is {\em su-$\beta$-H\"older}
if it is continuous and satisfies \eqref{FHolder} for points $x$ and $y$ in the same local stable and unstable leaves. This notion is also defined in the same way for a cocycle $\A$ on an 
su-$\beta$-H\"older subbundle $\E'$ of $\E$ by using local identifications $\I_{x,y}'$ on $\E'$.
\vskip.1cm

Finally, we define the notion of conjugacy between linear cocycles.

\begin{definition} \label{cohomology def E}
Let  $\A$ and $\B$  be  linear cocycles over $f$ on vector bundles $\E$ and $\E'$ over $\M$. Let 
 $\L=\L(\E,\E')$ be the bundle  whose fiber $\L_x$ is the space  $L(\E_x,\E'_x)$ of linear operators from
 $\E_x$ to $\E_x'$. A (measurable, continuous)  {\em conjugacy} $C$ between  $\A$ and $\B$  is a 
(measurable, continuous)  section of $\L$ taking values in invertible operators and satisfying 
equation \eqref{C def}.
%\begin{equation} \label{C def E} \,\B_x=C(f x) \circ \A_x \circ  C(x)^{-1}   \;\text{ for all } x\in \M.\end{equation}
\end{definition}

\vskip.1cm 
%%%%%%%%%%%%%   Holonomies %%%%%%%%%%%%%%%

\subsection{Holonomies and fiber bunching} \label{hol and FB}

An important role in the study of cocycles, and in this paper in particular,  is played by holonomies. Their existence was established for fiber bunched cocycles.

\begin{definition} \label{FB}
An su-$\beta$-H\"older linear cocycle $\A$ is called {\em fiber bunched} if there exist constants $\theta<1$ and $K$ such that for all $x\in\M$ and  $n\in \N$,
\begin{equation}\label{weak fiber bunched}
\| \A_x^n\|\cdot \|(\A_x^n)^{-1}\| \cdot  (\nu^n_x)^\beta < K\, \theta^n \;\text{ and}\quad
\| \A_x^{-n}\|\cdot \|(\A_x^{-n})^{-1}\| \cdot  (\hat \nu^{-n}_x)^\beta < K\, \theta^n,
\end{equation}
where 
$ \;\nu$ and $\hat \nu$ are  as in \eqref{partial def}.
\end{definition}

Existence of holonomies was proved for $GL(d,\R)$-valued cocycles  in \cite{AV,ASV} under a stronger fiber bunching assumption, and later extended to bundle setting in \cite{KS13} and to the weaker fiber bunching \eqref{weak fiber bunched} in \cite{S15}. The proofs apply to su-$\beta$-H\"older 
 cocycles  without modifications.

\begin{proposition}  \cite{AV,ASV,KS13,S15} \label{existence of holonomies} $\;$\\
Let $\A$ be an su-$\beta$-H\"older linear cocycle over a partialy hyperbolic diffeomorphism $f:\M \to \M$. If  $\A$ is fiber bunched, 
then for every $x\in \M$ and $y\in W^s(x)$ the limit 
\begin{equation}\label{holonomy def}
 H_{x,y}=H^{\A,s}_{x,y} =\underset{n\to\infty}{\lim} \,(\A^n_y)^{-1} \circ I_{f^nx,f^ny}\circ  \A^n_x ,
\end{equation}
called a {\em stable holonomy of $\A$}, \,exists and satisfies 
\begin{itemize}
\item[(H1)] $H_{x,y}$ is an invertible  linear map from $\E_x$ to $\E_y$;
\vskip.1cm
\item[(H2)]  $H_{x,x}=\Id\,$ and $\,H_{y,z} \circ H_{x,y}=H_{x,z}$,\,\,
which implies $(H_{x,y})^{-1}=H_{y,x};$
\vskip.1cm
\item[(H3)]  $H_{x,y}= (\A^n_y)^{-1}\circ H_{f^nx ,f^ny} \circ \A^n_x\;$ 
for all $n\in \N$;
\vskip.1cm
\item[(H4)] $\| H_{x,y} - I_{x,y} \,\| \leq c\,\dist (x,y)^{\beta},$
 where $c$ is independent of $x$   and $y\in W^s_{\text{loc}}(x).$
% and also  $\|(\A^n_y)^{-1} \circ \A^n_x - \Id \,\| \leq c\,\dist (x,y)^{\beta}\,$ for every $n\in \N$.
\vskip.1cm
\item[(H5)] The map $H^{\A,s}:\;(x,y)\mapsto H^{\A,s}_{x,\, y}$,\, where $x\in \M$ and  $y\in W^s_{\text{loc}}(x)$, is continuous.
\end{itemize}

\vskip.1cm

\noindent The {\em unstable holonomy} 
$$ H^{\A,u}_{x,y} =\underset{n\to\infty}{\lim} \,(\A^{-n}_y)^{-1} \circ I_{f^{-n}x,f^{-n}y}\circ \A^{-n}_x \;\,\text{ for }y\in W^u(x)$$
 also exists and satisfies  similar properties.
    \end{proposition}
    
 By  \cite[Proposition 4.2]{KS13}, for a fiber bunched su-$\beta$-H\"older cocycle the
 map $H$ satisfying (H1)-(H5) is unique. It follows that  the stable and unstable holonomies
  do not depend 
 on a particular choice of $\beta$-H\"older local identifications.

 \begin{definition} \label{def intertwine}
A conjugacy $C$ between 
$\A$ and $\B$  {\em intertwines their holonomies}  if
\begin{equation}\label{intertwines}
H_{x,y}^{\B,\,s/u}=C(y)\circ H_{x,y}^{\A,\,{s/u}}\circ C(x)^{-1}\quad \text{for all }x,y \in \M
\text{ with }y\in W^{s/u}(x).
\end{equation}
\end{definition}

%Linear cocycles over a dynamical system $f: \M \to \M$ appear naturally in various areas of dynamics and applications. Examples  range from derivative cocycles to random matrices. 

%In fact, the second term in the sum is not needed for continuous $\A$ and compact $\M$ as $(\A_x)^{-1}$ is then automatically continuous in $x$ and bounded on $\M$, so we can estimate $$ \|(\A_x)^{-1}-(\A_y)^{-1}\| =\| (\A_x)^{-1}\,(\A_y-\A_x)\,(\A_y)^{-1}\| \le K' \cdot \|\A_x-\A_y\|.$$

%%%%%%%%%%%%%%%%%%%%%%%%%%%
%%% Cohomology  %%%%%%%%
%%%%%%%%%%%%%%%%%%%%%%%%%%%

%%%%%%%%%%%%%%  Proof of Theorem 2.1 %%%%%%%%%
%%%%%%%%%%%%%%%%%%%%%%%%%%%%%%%%%%

\section{Twisted cohomological equation} \label{twisted}

In this section, $f$ is as in Assumptions \ref{PH}, $\E$ is a $\beta$-H\"older
 vector bundle over $\M$, or more generally an su-$\beta$-H\"older subbundle of a
 $\beta$-H\"older vector bundle over $\M$, %$\p$ and $\q$ are su-$\beta$-H\"older sections of $\E$,
 and  $\F$ is an su-$\beta$-H\"older linear cocycle  on $\E$ over $f$.
We study the cohomological  equation over $f$ twisted by $\F$ for sections of $\E$. 
We will use the main result of this section, Theorem \ref{twist-meas},
 in the inductive process in the proof of Theorem~\ref{main}.
\vskip.1cm

We say that a section $\p : \M \to \E$ is an {\em $\F$-twisted coboundary over $f$} if there exists 
a section $\q : \M \to \E$ satisfying the following twisted cohomological equation
\begin{equation}\label{twisteq}
\p(x)=\q(x) - (\F_x)^{-1} (\q(fx))\quad \text{equivalently}\quad \q(x)=\p(x) + (\F_x)^{-1} (\q(fx)).
\end{equation}

In Theorem \ref{twist-meas} we will establish regularity of a measurable solution of \eqref{twisteq} 
with uniformly bounded twist $\F$, and show its invariance under twisted holonomies, 
which we introduce below.

In the case of the trivial bundle $\E=\M \times \R^d$ and the trivial  twist $\F_x = \Id$, \eqref{twisteq}  is  the usual  vector-valued cohomological equation  $\p(x)=\q(x) - \q(fx)$.  In particular, Theorem \ref{twist-meas}
generalizes the usual measurable Livsic theorem for scalar cocycles in the hyperbolic case
and extends the corresponding partially hyperbolic result in \cite{W}.

\begin{definition} \label{slow}
We say that a linear cocycle $\A:\E\to\E$ is {\em dominated} if there exist numbers $\theta<1$ and $K$  
such that for all $x\in\M$ and $n\in \N$,
\begin{equation}\label{dominated}
\|(\F_x^n)^{-1}\| \cdot  (\nu^n_x)^\beta < K\, \theta^n \;\text{ and}\quad
 \|(\F_x^{-n})^{-1}\| \cdot  (\hat \nu^{-n}_x)^\beta < K\, \theta^n,
\end{equation}
where 
$\nu$ and  $\hat \nu$ are  as in \eqref{partial def}.
We say that $\A$ is {\em uniformly bounded}\, if there exists $K$
such that $\| \F_x^n \| \le K$ for all $x\in \M$ and $n\in \Z$.
\end{definition}

To study equation \eqref{twisteq} we consider the following twisted trajectory sum for $\p$:
\begin{equation}\label{PHI}
\P^n(x)=\, \p(x) + (\F_x)^{-1}(\p (fx))+\dots + (\F^{n-1}_{x})^{-1}(\p(f^{n-1}x)) \in \E_x.
\end{equation}

%%%%%
\begin{proposition} \label{twist hol}
Let $\p:\M\to\E$ be an su-$\beta$-H\"older section and let $\F:\E\to\E$ be an su-$\beta$-H\"older linear cocycle over $f$. Suppose that $\F$ is  dominated and fiber bunched, and let $\H_{y,x}^s=\H_{y,x}^{\F,s}$ be the stable holonomy for $\F$. Then the  limit
$$
\P_{y,x}^s= \P^{\F,\p,s}_{x,y}=\lim _{n \to \infty} (\P^n(x) - \H_{y,x}^s\P^n(y) )
%=\sum_{k=0}^\infty\, [\, (\F^{k}_{x})^{-1}(\p(f^{k}x))-  \H_{y,x}^s (\F^{k}_{y})^{-1}(\p(f^{k}y))\,]
$$
exists  for any $x\in \M$ and $y \in W^s(x)$ and satisfies 
\begin{itemize}
\item[($\P$1)] $\P^s_{y,x}\in \E_x$;
\vskip.1cm
\item[($\P$2)] $\P^s_{x,x}=0\,$ and $\,\P^s_{z,x}= \P^s_{y,x} +\H^s_{y,x}(\P^s_{z,y} )$;
\vskip.1cm
\item[($\P$3)] $\|\P^s_{y,x}\| \le K' d (x,y)^\beta$ where $K'$ is independent of $x\in \M$ and $y \in W^s_{loc}(x)$;
\vskip.1cm
\item[($\P$4)] The map $\P^{s}:\;(x,y)\mapsto \P^{s}_{y,x}$,\, where $x\in \M$ and  $y\in W^s_{\text{loc}}(x)$, is continuous.
 \end{itemize}

\noindent A similar result holds for 
$$\P_{y,x}^u= \lim _{n \to -\infty} (\P^n(x) - \H_{y,x}^u\P^n(y) ).$$
\end{proposition}

\begin{proof}
Using \eqref{PHI} we expand
$$
\P^n(x) - \H_{y,x}^s \P^n(y) =
\sum_{k=0}^{n-1}\,[(\F^{k}_{x})^{-1}(\p(f^{k}x))-
(\H_{y,x}^s \circ (\F^{k}_{y})^{-1} \circ \H_{f^kx,f^ky}^s) (\H_{f^ky,f^kx}^s\p(f^{k}y))].
$$
Since $\H_{y,x}^s \circ (\F^{k}_{y})^{-1} \circ \H_{f^kx,f^ky}^s= (\F^k_x)^{-1}$ by (H3),\, 
the $k^{th}$ term in the sum equals
  $$
(\F^{k}_{x})^{-1}[\p(f^{k}x))-  (\F^{k}_{x})^{-1}  (\H_{f^ky,f^kx}^s \p(f^{k}y)]=
(\F^{k}_{x})^{-1}\,[ \p(f^{k}x)-\H_{f^ky,f^kx}^s \p(f^{k}y)].
$$

For all $x\in \M$ and $y \in W^s_{loc}(x)$ we have $d(f^{k}x,f^{k}y)\le \nu_x^k \,d(x,y)$ for all $k\ge 0$. Since
 $\p$ is su-$\beta$-H\"older we have
 $$
 \|\p(f^{k}x)- \I_{f^ky,f^kx} \,\p(f^{k}y)\|\le K_1(\nu^k_x d(x,y))^\beta,
 $$
 and since $\H_{f^ky,f^kx}^s$ is $\beta$-H\"older close to $\I_{f^ky,f^kx}$ by (H4), we conclude that
 $$
 \|\p(f^{k}x)-  \H_{f^ky,f^kx}^s \,\p(f^{k}y)\|\le K_2(\nu^k_x\, d(x,y))^\beta.
 $$
Now using the first inequality in \eqref{dominated} we estimate
   $$
   \begin{aligned}
&\|(\F^{k}_{x})^{-1}\,[  \p(f^{k}x)   -\H_{f^ky,f^kx}^s\, \p(f^{k}y)] \,\|  \,\le\,
 \|(\F^{k}_{x})^{-1} \| \cdot \|\p(f^{k}x)-\H_{f^ky,f^kx}^s\,\p(f^{k}y)  \|  \\
 &\le\, \|(\F^{k}_{x})^{-1} \| \cdot  K_2(\nu^k_x\, d(x,y))^\beta \,\le\, K_2L\,\theta^{k} d(x,y)^\beta \quad\text{with }\theta<1.
 \end{aligned}
 $$
We conclude that the series 
$$
\sum_{k=0}^\infty\, [\, (\F^{k}_{x})^{-1}(\p(f^{k}x))-  \H_{y,x}^s (\F^{k}_{y})^{-1}(\p(f^{k}y))\,]
\,=\, \lim _{n \to \infty} (\P^n(x) - \H_{y,x}^s\P^n(y) )
$$
converges uniformly over all $x\in \M$ and $y\in W^s_{\text{loc}}(x)$.
This  yields   existence of $\P_{y,x}^s$ and property $(\P4)$. Further, we can estimate
$$
\| \P^n(x) - \H_{y,x}^s \P^n(y) \|  \le\, \sum_{k=0}^{n-1} K_2 L \, \theta^{k} d(x,y)^\beta \le K' d(x,y)^\beta,
$$
so that the limit satisfies $ \|\P_{x,y}^s\| \le K' d (x,y)^\beta$, which gives $(\P3)$. 
Property $(\P1)$ is trivial and $(\P2)$ follows by taking the limit in
$$ \P^n(x) - \H_{z,x}^s\P^n(z) = (\P^n(x) - \H_{y,x}^s\P^n(y) )+ (H_{y,x}^s\P^n(y) - H_{z,x}^s\P^n(z))=$$ 
$$ = (\P^n(x) - \H_{y,x}^s\P^n(y) )+ H_{y,x}^s(\P^n(y) - H_{z,y}^s\P^n(z)),$$
where we use $H^s_{z,x}=H_{y,x}^s\circ H^s_{z,y}$.
\end{proof}

%%%%
We now introduce  twisted holonomies,
which we then use  to analyze regularity of solutions of the twisted cohomological equation
 \eqref{twisteq}. These are the maps
 $$\h^s_{x,y}=\h^{\F,\p,s}_{x,y}: \E_x\to \E_y\;\text{  for  $y \in W^s(x)$.}$$
 
\begin{proposition}  \label{twist holonomies} 
Let $\p:\M\to\E$ be an su-$\beta$-H\"older section and let $\F:\E\to\E$ be an su-$\beta$-H\"older linear cocycle over $f$. Suppose that $\F$ is  dominated and fiber bunched, and let $\H_{x,y}^s$ and
$\P_{x,y}^s$ be as in Proposition \ref{twist hol}.  Then the maps
\begin{equation}\label{holonomy twist def}
 \h^s_{x,y}(v)=\H_{x,y}^s (v) +\P_{x,y}^s=\H_{x,y}^s (v) + \lim _{n \to \infty} (\P^n(y) - \H_{x,y}^s\P^n(x) )
\end{equation}
called\,  {\em stable twisted holonomies}, exist for any $x\in \M$ and $y \in W^s(x)$ and satisfy 
\begin{itemize}
\item[($\h$1)] $\h_{x,y}$ is an invertible affine map from $\E_x$ to $\E_y$;
\vskip.1cm
\item[($\h$2)]  $\h_{x,x}=\Id\,$ and $\,\h_{y,z} \circ \h_{x,y}=\h_{x,z}$;\,\,
\vskip.1cm
%\item[($\h$3)] $*d( \h_{x,y} ,\Id ) \leq c\,\dist (x,y)^{\beta},$ where $c$ is independent of $x$   and $y\in W^s_{\text{loc}}(x).$ \vskip.1cm
\item[($\h$3)] The map $\h^{\F,\p,s}:\;(x,y)\mapsto \h^{s}_{x,\, y}$\,  is continuous in $x\in \M$ and  $y\in W^s_{\text{loc}}(x)$.
\end{itemize}
\end{proposition}

\begin{proof}
This follows directly from the previous proposition. For ($\h$2) we use $(\P2)$:
$$\,\h_{y,z} \circ \h_{x,y}(v)=\h_{y,z} (\H_{x,y}^s (v) +\P_{x,y}^s)=
(\H_{y,z} \circ \H_{x,y}^s) (v) + \H_{y,z}(\P_{x,y}^s)+\P_{y,z}^s=$$
$$=\H_{x,z}^s (v) + \P_{x,z}^s=\h_{x,z}.$$
\end{proof}

Now we formulate and prove the main result of this section.

\begin{theorem} \label{twist-meas}
Let $(f,\mu)$ be as in Assumptions \ref{PH} and let $\E$ be an su-$\beta$-H\"older
subbundle of a $\beta$-H\"older vector bundle over $\M$.  Let $\F:\E\to \E$ be an  
su-$\beta$-H\"older uniformly bounded cocycle over $f$.
Let $\p  : \M \to \E$ be an su-$\beta$-H\"older section, and let $\q:\M\to \E$ be a 
$\mu$-measurable section satisfying 
$$\p(x)=\q(x) - (\F_x)^{-1} (\q(fx)).$$
Then $\q$ is  su-$\beta$-H\"older and invariant under the twisted holonomies, that is,
$$
\q(y)=\h^{\F,\p,s}_{x,y} \,\q(x) \quad \text{for all $x\in X$ and  $y \in W^s(x)$}.
$$

\end{theorem}

\begin{proof}
Clearly, an su-$\beta$-H\"older uniformly bounded cocycle is both dominated and fiber-bunched.
Hence $\F$ has  holonomies, %$\H^s_{x,y}:\E_x \to \E_y$ where $y \in W^s(x)$ 
and Propositions \ref{twist hol} and \ref{twist holonomies} yield $\P_{x,y}^s=\P^{\F,\p,s}_{x,y}$ and twisted holinomies $\h_{x,y}^s=\h^{\F,\p,s}_{x,y}$.
Iterating \eqref{twisteq} we obtain
$$
\begin{aligned}
 \q(x)& =\p(x) + (\F_x)^{-1} (\q(fx))= \p(x) + (\F_x)^{-1}[\p (fx)+\F_{fx} (\q(f^2x))]=\dots \\
  &= \p(x) + (\F_x)^{-1}(\p (fx))+\dots + (\F^{n-1}_{x})^{-1}(\p(f^{n-1}x))+  (\F_x^n)^{-1} (\q(f^{n}x))\\
  &= \P^n(x)+ (\F_x^n)^{-1} (\q(f^{n}x)).
\end{aligned}
$$

Let $y\in \M$ and $x\in W^s(y)$. Using the equation above for $\q(x)$ and $\q(y)$ we obtain
$$
\q(x)-\H_{y,x}^s \,\q(y)=
\P^n(x) - \H_{y,x}^s \P^n(y) + \Delta_n, \quad \text{where}
$$
$$
 \Delta_n= (\F_x^n)^{-1} (\q(f^{n}x)) -\H_{y,x}^s (\F_y^n)^{-1} (\q(f^{n}y)).
$$
By Proposition \ref{twist hol},  $(\P^n(x) - \H_{y,x}^s\P^n(y) )$ converges to $\P_{y,x}^s$.
\vskip.1cm

Now we show that $\|\Delta_n\|\to 0$ along a subsequence for all $x,y$ in a set of full measure.
First we note that by property (H3) we have
 $\H_{y,x}^s \circ (\F_y^n)^{-1}=(\F_x^n)^{-1} \circ \H_{f^ny,f^nx}^s$. Hence
$$
\Delta_n=
(\F_x^n)^{-1} \left( \q(f^{n}x) -\H_{f^ny,f^nx}^s(\q(f^{n}y)) \right)= (\F_x^n)^{-1} ( \Delta_n'),
$$
where  $\Delta_n'=\q(f^{n}x) -\H_{f^ny,f^nx}^s(\q(f^{n}y))$. By uniform boundedness of $\F$  we obtain
$$
\| \Delta_n\|\le \|(\F_x^n)^{-1} \| \cdot \|\Delta_n'\|\le K  \|\Delta_n'\|.
$$
Since the section $\q:\M\to E$ is $\mu$-measurable, by Lusin's theorem there
exists a compact set $S\subset \M$ with $\mu(S)>1/2$ such that
$\q$ is uniformly continuous and hence bounded on  $S$. Let $Y$ be the set of points in $\M$
for which the frequency of visiting $S$ equals $\mu(S)$.
By Birkhoff Ergodic Theorem, $\mu(Y)=1$.

If $x,y\in Y$, there exists a subsequence $n_i\to \infty$ such that $f^{n_i}x, f^{n_i}y \in S$ for all $i$.
Since  $y \in W^s(x)$, $\,d(f^{n_i}x, f^{n_i}y)\to 0$
 and hence $\|\Delta_{n_i}' \|\to 0$ by uniform continuity and boundedness
of $\q$ on $S$ and property (H4) of $\H^s$. Thus $\|\Delta_{n_i}\| \to 0$ and we obtain 
$$
\q(x)=\H_{y,x}^s \,\q(y) +\P_{y,x}^s=\h^s_{y,x} (\q(y)) \quad\text{ for all $x,y\in Y$ with $x \in W^s(y).$}
$$
Since $\P_{x,y}^s$ and $\H_{x,y}^s$ are $\beta$-H\"older
on $W^s_{\loc}(x)$ by $(\P3)$ and (H4) respectively, we get
$$
\|\q(x)- \I_{y,x}\q(y)\| \le \| (\H_{y,x}^s- \I_{y,x})\q(y)\| +\|\P_{y,x}^s\|\le K'\|\q(y)\| \,d(x,y)^\beta
$$
for all $x,y\in Y$ with $x \in W^s(y).$
This means that $\q$ is essentially stable holonomy invariant and essentially $\beta$-H\"older along 
$ W^s_{\loc}(y)$. Once we show that $\q$ is continuous on $\M$, and hence bounded, this will 
yield that $\q$ is s-$\beta$-H\"older and stable holonomy invariant.
Similar arguments show that for all $x,y\in Y$ with $x \in W^u(y)$ we also have
$$
\q(x)=\H_{y,x}^u \,\q(y) +\P_{y,x}^u=\h^u_{y,x} (\q(y)), 
$$
that is, $\q$ is also essentially unstable holonomy invariant and essentially $\beta$-H\"older 
along $ W^u_{\loc}(x)$. 

To complete the proof it remains to establish global continuity of $\q$ on $\M$.
We will use the following results from \cite{ASV}, which we formulate  using our notations.

\begin{definition}\cite[Definition 2.9]{ASV}  \label{ASV1}
Let $(\M,f)$ be a partially hyperbolic system,
and let $\n$ be a continuous fiber bundle over $\M$.
 A {\em stable holonomy} on $\n$ is a family of $\beta$-H\"older homeomorphisms $h^s_{x,y} : \n_x\to \n_y$ with uniform $\beta>0$, 
 defined for all $x, y$ in the same stable leaf of $f$ and satisfying
\begin{itemize}
\item[(a)]  $h^s_{y,z} \circ h^s_{x,y} = h^s_{x,z}$ and $h^s_{x,x} = \Id$,

\item[(b)]   the map $(x, y, \eta )\mapsto h^s_{x,y}(\eta)$  is continuous when $(x, y)$ varies 
in the set of pairs of  points in the same local stable leaf.
\end{itemize}
\end{definition}
\noindent Unstable holonomy is defined similarly, for pairs of points in the same unstable leaf.

We take $\n=\E$ and the holonomy maps  $h^s_{x,y}=\h^s_{x,y}: \E_x\to \E_y$. 
 Properties ($\h$2) and ($\h$3) of Proposition \ref{twist holonomies} yield 
 Properties (a) and (b) of the definition.
 By ($\h$1),  maps  $\h^s_{x,y}$ are invertible affine and hence are Lipschitz homeomorphisms.  
The argument above shows that section $\q$ is bi-essentially invariant in the following sense.

\begin{definition}\cite[Definition 2.10]{ASV} \label{ASV2}
A measurable section $\Psi : \M\to \n$ of the fiber bundle $\n$
 is called $s$-invariant if $h^s_{x,y} (\Psi(x))= \Psi(y)$
 for every $x, y$ in the same stable leaf
and {\em essentially $s$-invariant} if this relation holds restricted to some full measure subset. 
The definition of $u$-invariance is analogous. Finally, $\Psi$ is {\em bi-invariant} if it is both 
$s$-invariant and $u$-invariant, and it is {\em bi-essentially invariant} if it is both essentially $s$-invariant 
and essentially $u$-invariant.
\end{definition}

 A set in $\M$ is called {\em bi-saturated} if it consists of full stable and unstable leaves.
 By remark after Definition 2.10 in \cite{ASV} every Hausdorff topological space with a countable 
basis of topology is {\em refinable} and thus 
 Theorem D below applies and yields  that, up to modification on a set of measure zero, 
 $\q$ is continuous on  $\M$.

\begin{theorem}\cite[Theorem D]{ASV}  \label{ASVD}
 Let $f:\M\to \M$ be a $C^2$ partially hyperbolic 
center bunched  diffeomorphism preserving a volume $\mu$, and let $\n$ be 
a continuous fiber bundle with stable and unstable holonomies and with refinable fiber. 
Then, 
\begin{itemize}
\item[(a)] for every bi-essentially invariant section $\Psi : \M\to \n$, there exists a bi-saturated 
set $\M_\Psi$ with full measure, and a bi-invariant section $\tilde \Psi : \M_\Psi \to \n$ that 
coincides with $\Psi$ at $\mu$ almost every point.
\vskip.1cm
\item[(b)]  if $f$ is  accessible then $\M_\Psi=\M$ and $\tilde \Psi$ is continuous.
\end{itemize}
\end{theorem}

\noindent This completes the proof  of Theorem \ref{twist-meas}.  \end{proof}

 %%%%%%%%%%%%
 %%%%%%%%%%%%
 %%%%%%%%%%%%
 
 \section{Proof of Theorem \ref{main}} \label{cocycle proofs}

 %%%%%%%%%%%%
\subsection{Continuity of measurable conjugacy in uniformly quasiconformal case} 
\label{proof of thm QC}
An important ingredient in the proof of Theorem \ref{main} is the following result, 
which extends \cite[Theorem 4.2]{KS15}. We recall that a cocycle $\A$ is 
{\em uniformly quasiconformal}  if $\,\| \A^n_x\|\cdot \|( \A^n_x)^{-1}\|$ is uniformly bounded in $x\in \M$ and  $n\in \Z$.  

\begin{theorem}   \label{measurable cohomology}  
Let $(f,\mu)$ be as in Assumptions \ref{PH}.
\vskip.1cm 
\noindent {\bf (i) Continuous version.}  Let $\E$ and $\E'$ be continuous vector bundles 
over $\M$, and let  $\A$ and $\B$ be continuous linear cocycles over $f$ on $\E$ and $\E'$ respectively.
 Suppose that $\A$ and $\B$ have stable and unstable holonomies satisfying (H1,2,3,5) of Proposition \ref{existence of holonomies}, and $\A$ is uniformly quasiconformal.
Then any $\mu$-measurable conjugacy between $\A$ and $\B$ coincides on a set of full 
measure with a continuous conjugacy which intertwines the  holonomies of $\A$ and $\B$.
%Moreover, if the  holonomies also satisfy (H4) then $C$ is su-$\beta$-H\"older.
\vskip.1cm 

\noindent {\bf (ii) su-H\"older version.}
 Let $\E$ and $\E'$ be 
%$\beta$-H\"older vector bundles over $\M$, or more generally 
su-$\beta$-H\"older subbundles of $\beta$-H\"older vector bundles 
over $\M$. Let $\A$ be a uniformly quasiconformal su-$\beta$-H\"older linear cocycle over $f$ on $\E$.
Let $\B$ be an su-$\beta$-H\"older  fiber bunched linear cocycle over $f$ on $\E'$  or, more generally, a continuous  linear cocycle with holonomies as in Proposition \ref{existence of holonomies}.
Then any $\mu$-measurable conjugacy between $\A$ and $\B$ coincides on a set of full measure 
with an su-$\beta$-H\"older conjugacy.
\end{theorem}

\begin{proof}
Recall that $\L=\L(\E,\E')$ is the vector bundle with fiber $\L_x=L(\E_x,\E'_x)$.
Let $C$ be a $\mu$-measurable conjugacy between $\A$ and $\B$, that is, a 
$\mu$-measurable section  of $\L$ taking values in invertible linear operators 
and satisfying  
$$
\B_x=C(f x) \circ \A_x \circ  C(x)^{-1} \;\text{ for  $\mu$ almost every $x$.}
$$
The main part of the proof is  showing that $C$ intertwines the stable holonomies of $\A$ and $\B$ on a set of full measure.

 Since $C$ is $\mu$-measurable and the bundle $\L$ has countable 
basis of topology, by Lusin's theorem there 
exists a compact set $S\subset \M$ with $\mu(S)>1/2$ such that 
$C$ is uniformly continuous on  $S$. 
%It follows that $\| C\|$ and $\|C^{-1} \|$ are bounded on $S$.
Let $Y$ be the set of points in $\M$
for which the frequency of visiting $S$ equals $\mu(S)$.
By Birkhoff ergodic theorem, $\mu(Y)=1$. %and is  $f$-invariant
\vskip.1cm

Suppose that $x,y\in Y$ and  $y \in W^s(x)$. Then 
\begin{equation} \label{meascoh1}
\begin{aligned}
&(\B^n_y)^{-1}\circ \I_{f^{n}x f^{n}y}\circ \B^n_x= \\
&=\left( C(f^ny)\circ \A^n_y\circ C(y)^{-1}\right)^{-1} \circ  \I_{f^{n}x, f^{n}y} \circ
C(f^nx)\circ  \A^n_x\circ  C(x)^{-1} \\
&= C(y) \circ (\A^n_y)^{-1}\circ  C(f^n y)^{-1} \circ  \I_{f^{n}x, f^{n}y}\circ 
 C(f^nx)\circ \A^n_x \circ C(x)^{-1}= \\
&= C(y) \circ (\A^n_y)^{-1}  \circ ( \I_{f^{n}x, f^{n}y} +\Delta_n) \circ \A^n_v \circ C(x)^{-1}= \\
&= C(y) \circ (\A^n_y)^{-1}  \I_{f^{n}x ,f^{n}y} \circ \A^n_x\circ C(x)^{-1} + 
C(y)\circ (\A^n_y)^{-1} \circ \Delta_n \circ \A^n_x \circ C(x)^{-1} . 
\end{aligned}
\end{equation}
We will show that the last term tends to 0 along a subsequence
 $\{n_i\}$ such that $f^{n_i}x, f^{n_i}y \in S$ for all $i$.
Since $x,y \in Y$, such a subsequence exists by the choice of $Y$.
First we note that for the map 
$$\Delta_n=C(f^n y)^{-1} \circ  \I_{f^{n}x, f^{n}y}\circ 
 C(f^nx)-\I_{f^{n}x, f^{n}y}: \,\,\E_{f^nx}\to \E_{f^ny}
 $$
we have $\| \Delta_{n_i} \| \to 0$ as $i\to\infty$ since $\,\dist (f^{n_i}x, f^{n_i}y)\to 0$
for $y \in W^s(x)$.
%\begin{equation} \label{Delta_n}  \| \Delta_n \|= \| C(f^n y)^{-1} \circ C(f^n x) - \Id\| \le  \| C(f^n y)^{-1} \| \cdot \|C(f^n x) - C(f^n y)\|. \end{equation}
% $\|C(f^{n_i} x) - C(f^{n_i} y)\| \to 0$
This follows from uniform continuity of $C$ on the compact set $S$. 

\vskip.1cm

Since the  norms of stable holonomies are uniformly bounded 
 over pairs of points in local stable leaves, and since $\A$ is uniformly quasiconformal, we obtain %\begin{equation}\label{B}
$$
\begin{aligned}
&\| (\A^n_y)^{-1}\| \cdot \| \A^n_x \| \,\le\,   \| (\A^n_y)^{-1}\| \cdot \|H_{f^ny,f^nx}^{\A,s}\| \cdot \| \A^n_y \| \cdot \|H_{x,y}^{\A,s}\| \,\le \\
& \| (\A^n_y)^{-1}\| \cdot K_1 \| \A^n_y \|
\,\le\,  K_1 \, K_2 \quad\text{for all $x\in \M$ and $y\in W^s_{\loc}(x)$.}
\end{aligned}
$$
%\end{equation}
 Now it follows that 
$$
\| C(y)\circ (\A^{n_i}_y)^{-1} \circ \Delta_{n_i} \circ \A^{n_i}_x \circ C(x)^{-1} \| \le K_1 \, K_2 \,  \| \Delta_{n_i} \| 
\to 0 \quad\text{ as }i\to\infty.
$$
Passing to the limit in \eqref{meascoh1}  along the sequence $n_i$ 
we obtain that $C$ intertwines the stable holonomies $H^{\A,s}$ and $H^{\B,s}$ 
on a set of full measure:
\begin{equation}\label{full measure}
H_{x,y}^{\B,s}=C(y)\circ H_{x,y}^{\A,s} \circ C(x)^{-1}
\quad\text{for all } x,y\in Y \text{ such that }y\in W^s(x).
\end{equation}
or equivalently
\begin{equation}\label{full measure2}
C(y) =H_{x,y}^{\A,s}\circ C(x)\circ (H_{x,y}^{\B,s})^{-1}
\quad\text{for all } x,y\in Y \text{ such that }y\in W^s(x).
\end{equation}
Similarly, we  obtain that $C$ intertwines the unstable holonomies $H^{\A,u}$ and $H^{\B,u}$ 
on a set of full measure.
Together these imply that $C$ is a bi-essentially invariant section, in the sense of Definition \ref{ASV2},
of the bundle $\n=\L$ with stable holonomy maps  $h^s_{x,y}: \L_x\to \L_y$  defined as
$$ h^s_{x,y} (C)=H_{x,y}^{\A,s}\circ C\circ (H_{x,y}^{\B,s})^{-1}$$
and similarly defined unstable holonomies. Properties (H2) and (H5) of Proposition \ref{existence of holonomies} imply that these holonomies satisfy Properties (a) and (b) of Definition \ref{ASV1}.
Also, the maps  $h^s_{x,y}$ are  invertible linear and hence are Lipschitz homeomorphisms.  
 Since the space $\L_x=L(\E_x,\E'_x)$ is Hausdorff  with a countable 
basis of topology, it is refinable and thus Theorem \ref{ASVD} applies and yields that, 
up to modification on a set of measure zero, $C$ is continuous on  $\M$. 
Now \eqref{full measure} shows that $C$ intertwines the holonomies of $\A$ and $\B$ 
everywhere on $\M$. This completes the proof of the first part of the theorem.
\vskip.2cm 
In the second part, since $\A$ is su-$\beta$-H\"older, uniform quasiconformality gives fiber bunching,
and hence existence of holonomies by Proposition \ref{existence of holonomies}. The cocycle $\B$
also has holonomies by Proposition \ref{existence of holonomies} or by the assumption. Thus the
first part applies and yields that the conjugacy $C$ is continuous and intertwines the  holonomies of $\A$ and $\B$. The latter means that  \eqref{full measure2} holds everywhere,  and it follows  that $C$ is s-$\beta$-H\"older. Indeed, Proposition \ref{existence of holonomies} (H4) gives $\beta$-H\"older continuity
 of $H^{\A,s}$ and $H^{\B,s}$ along $W^s$, which yields that of $C$. Similarly,  $C$ is also 
 u-$\beta$-H\"older and thus su-$\beta$-H\"older.
\end{proof}

\subsection{Regularity results for measurable invariant structures.}
\label{invariant} $\;$\\
In this section we give more general versions of earlier results on regularity of measurable 
invariant subbudles, Riemannian metrics, and conformal structures for linear cocycles.
We will use these results in the proof of Theorem \ref{main}
 
We denote by  $\lambda_+(\A,\mu)$ and $\lambda_-(\A,\mu)$ the largest and smallest Lyapunov 
exponents of a linear cocycle $\A$ with respect to $\mu$, given by the Oseledets Multiplicative Ergodic Theorem. 
For  $\mu$ almost all  $x\in \M$, they equal to the following limits
\begin{equation} \label{exponents}
\lambda_+(\A,\mu)= \lim_{n \to \infty} n^{-1} \ln \| \A_x ^n \| 
\quad \text{and}\quad
\lambda_-(\A,\mu)=  \lim_{n \to \infty} n^{-1}  \ln \| (\A_x ^n)^{-1} \|^{-1}. 
\end{equation}
A cocycle $\A$ {\em has one exponent}\, with respect to $\mu$ if 
$\lambda_+(\A,\mu)=\lambda_-(\A,\mu)$. We note that a cocycle with more than one Lyapunov exponent may have measurable invariant  sub-bundles which are not continuous. In particular, the Lyapunov 
sub-bundle for the negative Lyapunov exponent of cocycle $\A$ as in \eqref{PWex} is measurable
but not continuous, see \cite[Example 2.9]{S13}.
%which is directly related to discontinuity of the measurable conjugacy to $\B$. 
In contrast, for cocycles with one Lyapunov exponent we have

\begin{theorem} \label{distribution}
Let $(f,\mu)$ be as in Assumptions \ref{PH}, let $\E$ be an su-$\beta$-H\"older
subbundle of a $\beta$-H\"older vector bundle over $\M$, and let $\B$ be a fiber bunched su-$\beta$-H\"older linear cocycle over $f$ on $\E$. If  $\lambda_+(\B,\mu)=\lambda_-(\B,\mu)$ then any $\mu$-measurable $\B$-invariant subbundle $\E'$ of $\E$  coincides $\mu$-a.e. with an su-$\beta$-H\"older
sub-bundle invariant under $\B$ and under its holonomies.
\end{theorem}

This is essentially  \cite[Theorem 3.3]{KS13} with $\beta$-H\"older assumption on $\B$ weakened
to su-$\beta$-H\"older and continuity of $\E'$ improved to su-$\beta$-H\"older in the conclusion.

\begin{proof}
 The proof of Theorem 3.3 in \cite{KS13} goes through essentially without change as it relies on
Theorem C in  \cite{ASV} to show holonomy invariance and continuity of $\E'$.  Theorem C in  \cite{ASV}
 requires only $\lambda_+(\B,\mu)=\lambda_-(\B,\mu)$, continuity of $\B$, and existence of 
 holonomies, for which it suffices to have $\B$ fiber bunched and su-$\beta$-H\"older.
Then holonomy invariance of $\E'$ and H\"older property (H4) of holonomies along  $W^s$ and $W^u$
yield that $\E'$ is su-$\beta$-H\"older.
\end{proof}

\begin{theorem} \label{structure}
Let $(f,\mu)$ be as in Assumptions \ref{PH} and let $\E$ be an su-$\beta$-H\"older
subbundle of a $\beta$-H\"older vector bundle over $\M$.
Let  $\B$ be either a fiber bunched  su-$\beta$-H\"older linear cocycle over $f$ on $\E$,
or, more generally,  a continuous linear cocycle with holonomies as in Proposition \ref{existence of holonomies}.
Then  any $\B$-invariant $\mu$-measurable Riemannian metric (resp. conformal structure) on $\E$
coincides $\mu$-a.e. with an su-$\beta$-H\"older Riemannian metric (resp. conformal structure)
 invariant under $\B$ and under its holonomies.% on $\E$.
\end{theorem}

We recall that the space $\T$ of inner products on $\R^d$ identifies with the space of real 
symmetric positive definite $d\times d$ matrices, which is isomorphic to $GL(d,\R) /SO(d,\R)$. 
The group $GL(d,\R)$ acts transitively on $\T$ via $A[D] = A^T D \, A,$ where $A\in GL(d,\R)$ 
and $D \in \T.$ The space $\T$ is a Riemannian symmetric space of non-positive curvature 
when equipped with a certain $GL(d,\R)$-invariant metric \cite[Ch.\,XII, Theorem 1.2]{L}. 
%Using the background Riemannian metric on $\E$, we can identify an inner product with a symmetric linear operator. 
A conformal structure on $\R^d$, $d\geq 2$, is a class of proportional inner products. 
The space of conformal structures on $\R^d$ can be similarly identified with $SL(d,\R) /SO(d,\R)$,
which is also a Riemannian symmetric space of non-positive curvature 
with a $GL(d,\R)$-invariant metric.
%GL(d,\R)$ acts transitively on $\K$ via $X[K] = (\det X^TX)^{-1/d}\; X^T K \, X, $ and $\c^d$ carries a $GL(d,\R)$-invariant Riemannian metric of non-positive curvature.  
Riemannian metric (resp. conformal structure) on a vector bundle $\E$ is a section of 
the corresponding bundle whose fiber  at $x$ is the  space  of inner products  (resp. conformal structures)
on $\E_x$. See \cite{KS10} for more details.

%%%%%%%%%%%%%%%%%%%%%
\vskip.2cm 
\noindent{\it Proof of Theorem \ref {structure}.}
For a fiber bunched $\beta$-H\"older cocycle, global continuity of an invariant $\mu$-measurable conformal structure was established in \cite[Theorem 3.1]{KS13}. The main step, \cite[Proposition 4.4]{KS13},
proves essential holonomy invariance of the conformal structure. Fiber bunching and $\beta$-H\"older
continuity are used only to obtain holonomies, and thus they can be replaced by assuming existence of 
holonomies or by fiber bunching and the su-$\beta$-H\"older property, which imply it.
The global continuity and holonomy invariance, together with the H\"older property  (H4) of holonomies 
along  $W^s$ and $ W^u$, yield that  the conformal structure is su-$\beta$-H\"older. 
This completes the proof in the conformal structure case.

The proof for a Riemannian metric is almost identical, using the space of 
inner products in place of the space of conformal structures, which have the same properties 
for the purpose of the proof, described above. Alternatively, the result can be deduced by obtaining
an invariant conformal structure using the previous case, and then using boundedness
of the cocycle to find a proper normalization.
%Alternatively, this case can be deduced by first applying the conformal structure result and normalizing the structure to obtain a Riemannian metric $\tau$ on $\E$ for which $\B$ is conformal. That is, $\B \tau = b \tau$ for some positive su-$\beta$-H\"older function $b$, which generates the associated  $GL(1,\R)$ cocycle. Finding an invariant metric in the form  $\theta \tau$ is equivalent to solving scalar coboundary equation $b(x)=\theta(fx) \theta(x)^{-1}$, this is cohomoligy of  cocycle $b$ to constant $1$ cocycle. (Equivalently, we have additive equation for logarithms). The uniform boundedness of cocycle $\B$ implies that the orbit products $b^n(x)$ are uniformly bounded in $GL(1,\R)$. This gives a bounded measurable solution $\theta=\limsup_{n \to \infty}b^n(x)$ (in a sense, this is a one-dimensional version of the argument in the proof of the next corollary). Such $\theta$ is then su-$\beta$-H\"older by the regularity results for scalar cocycles in \cite{W}, or by the  one-dimensional case of Theorem~\ref{measurable cohomology}.
$\QED$

\begin{corollary} \label{uniform preserves}
Let $(f,\mu)$ be as in Assumptions \ref{PH} and let $\E$ be an su-$\beta$-H\"older
subbundle of a $\beta$-H\"older vector bundle over $\M$.
Suppose that $\B$\, is either an su-$\beta$-H\"older linear cocycle or a continuous linear cocycle with holonomies as in Proposition~\ref{existence of holonomies}. If $\B$ is uniformly bounded (resp. uniformly quasiconformal) 
then $\B$ preserves an su-$\beta$-H\"older invariant Riemannian metric (resp. conformal structure) on $\E$ invariant under the holonomies of $\B$.
\end{corollary}

\begin{proof}
We note that for an su-$\beta$-H\"older cocycle both uniform boundedness and uniform quasiconformality imply fiber bunching and give existence of holonomies. 
Thus $\B$ has holonomies and by the previous theorem it suffices to obtain an invariant measurable Riemannian metric (resp. conformal structure) on $\E$. For conformal structure this is \cite[Proposition 2.4]{KS10}. The same result holds
for the case of a Riemannian metric and argument carries over without changes. %For a continuous  Riemannian  metric $\tau$ on $\E$ we consider the set of pullbacks of the inner products along the orbit to $\E_x$: $S(x)=\{\,(\B^n_{x})^*(\tau_{f^n (x)}): \;n\in\Z\,\} .$ Since the cocycle $\B$ is uniformly bounded, the sets $S(x)$ have uniformly bounded diameters in the space of inner products. Since this space has non-positive curvature,  for every $x$ there exists a unique smallest closed ball containing $S(x)$ \cite[Ch.\,XI, Theorem 3.1]{L}. Its center $\hat \tau_{x}$ gives a (bounded) $\B$-invariant Riemannian  metric which is Borel measurable as the pointwise limit of continuous metrics $\tau^k_{x}$ given by the centers of the finite sets $S^k(x)=\{\,(\B^n_{x})^*(\tau_{f^n (x)}): \;|n| \le k \},$ which converge to $S(x)$ in Hausdorff distance.
\end{proof}

%%%%%%%%%%%%%  
 \subsection{Proof of Theorem \ref{main}} \label{proof of thm main}$\;$\\
We consider the invariant flag  \eqref{flagH} for $\A$ assumed in the theorem,
$$
\{0\} =\vv^0 \subset \vv^1 \subset \dots   \subset \vv^{k-1} \subset \vv^k = \E,
\; \text{ and the quotient-bundles $\,\U^i=\vv^i/\vv^{i-1}$. }
$$
We fix a background su-$\beta$-H\"older Riemannian metric $g$ on $\E$. Then for $i=1,\dots,k$, 
 the orthogonal complement of $\vv^{i-1}$ in $\vv^i$ is an su-$\beta$-H\"older subbundle of $\E$,
 which we denote by $\V^i$. Then we have $\vv^i=\V^1 \oplus \cdots \oplus \V^i$, but  in general only
  $\V^1 =\vv^1$ is $\A$-invariant while $\V^i$ with $i>1$ are not.

 We use the splitting $\E=\V^1 \oplus \cdots \oplus \V^k$ to define a block  triangular structure for $\A$.
 We denote by $P^{j}:\E \to \V^j$ the projection to the $\V^j$  component in  this splitting, and 
 define the blocks $\A^{j,i} : \V^i \to \V^j$ as $\A^{j,i} =P^j\circ \A |_{ \V^i}$. 
The invariance of the flag implies that $\A^{j,i}=0$ for $j>i$.

The projection $P^{i}|_{ \vv^i}$ induces a continuous bundle isomorphism between
 $\V^i$ and the quotient $\U^i$. Here we use continuous category for quotient bundles 
 and structures on them, since the  su-$\beta$-H\"older regularity was defined only for
 subbundles. However, using this isomorphism we identify the quotient cocycle $\tilde \A^{(i)}$
 on $\U^i$ with a linear cocycle $\A^{(i)}$ on the su-$\beta$-H\"older subbundle $\V^i$. 
 The cocycle $\A^{(i)}$ is also su-$\beta$-H\"older, as $\A^{(i)}_x$  coincides with  
 the block $\A^{i,i}_x$. By continuity of the isomorphism, since $\psi \tilde \A^{(i)}$ is
 uniformly bounded by the assumption, so is the cocycle $\psi \A^{(i)}$.
Now Corollary~\ref{uniform preserves} yields that $\psi \A^{(i)}$ preserves 
an su-$\beta$-H\"older Riemannian metric $\sigma_i$ on $\V^i$.
By isomorphism,  the quotient $\psi \tilde \A^{(i)}$ also preserves a continuous Riemannian metric 
$\tilde \sigma_i$ on $\U^i$. %Now we have 

\begin{proposition}  \label{prop distortion} 
Let $\A$ and $k$ be as in Theorem \ref{main}. Then there exists a constant $c$ such that 
 for all $x\in \M$ and $\,0\ne n\in \Z$,
$$\|(\psi \A)_x^n\| \le c |n|^{k-1} \quad \text{and } \quad \|\A_x^n\|\cdot \|( \A_x^n)^{-1}\| \le c |n|^{2(k-1)}.
$$
 In particular, $\A$ has one Lyapunov exponent  for each 
  $f$-invariant ergodic measure $\nu$, that is $\lambda_+(\A,\nu)=\lambda_-(\A,\nu)$.
\end{proposition}

\begin{proof}
This follows from the proof of \cite[Theorem 3.10]{KS13}  which uses only the invariant
 flag  \eqref{flagH} with continuous  invariant Riemannian metrics on the quotients. The last statement
 follows from the second inequality since
\vskip.2cm
\hskip1cm$ \lambda_+(\A,\nu)-\lambda_-(\A,\nu)=\underset {n\to\infty}\lim n^{-1} \ln (\|\A_x^n\| \cdot \|( \A_x^n)^{-1}\|)\quad\text{for $\nu$ a.e. $x$.}
 $ 
 \end{proof}

\vskip.1cm
Since we have $\lambda_+(\A,\mu)=\lambda_-(\A,\mu)$ and $\B$ is $\mu$-measurably 
conjugate to $\A$, we also have $\lambda_+(\B,\mu)=\lambda_-(\B,\mu).$ 
This follows from an easy lemma: 

\begin{lemma} \cite[Lemma 4.4]{KSW} \label{equal exp}
Let $\mu$ be an ergodic $f$-invariant measure.
If $\c$ is a $\mu$-measurable conjugacy between cocycles $\A$ and $\B$, then for $\mu$ a.e. $x$
and for each vector $0\ne u \in \E_x$ the forward (resp. backward) Lyapunov exponent of $u$
under $\A$ equals that of $\c(x)u$ under $\B$.
\end{lemma}

Now we construct the corresponding flag structure for $\B$. Denoting $\mathcal{V}^i_x=\c(x) \vv^i_x$ 
we obtain  the corresponding flag of measurable $\B$-invariant  sub-bundles
$$
\{0\} =\mathcal{V}^0 \subset \mathcal{V}^1\subset \mathcal{V}^2 \subset \dots \subset \mathcal{V}^k=\E'.
$$
Since $\lambda_+(\B,\mu)=\lambda_-(\B,\mu)$, and since $\B$ is su-$\beta$-H\"older 
and fiber bunched by the assumption, Theorem \ref{distribution} yields that this flag 
for $\B$ is su-$\beta$-H\"older. 

Similarly to the case of $\A$,  for each $i=1, ... , k$, we define the corresponding objects for $\B$: 
the continuous quotient bundle $\UU^i=\mathcal{V}^i/\mathcal{V}^{i-1}$ with the induced quotient 
cocycle $\tilde \B^{(i)}$;   the orthogonal complement  $\VV^i$ 
of $\mathcal{V}^{i-1}$ in $\mathcal{V}^i$ with respect to an su-$\beta$-H\"older background 
Riemannian metric $g'$ on $\E'$; the projection $\PP^{i}:\EE \to \VV^i$ for the su-$\beta$-H\"older splitting
$\E'=\VV^1 \oplus \cdots \oplus \VV^k$; su-$\beta$-H\"older blocks $\B^{j,i} =\PP^j\circ \B |_{ \V^i}$ 
with  triangular structure $\B^{j,i}=0$ for $j>i$; and the su-$\beta$-H\"older cocycle $\B^{(i)}$ on  $\VV^i$
with $\B^{(i)}_x=\B^{i,i}_x$ that is continuously  isomorphic to the quotient cocycle $\tilde \B^{(i)}$.

We note that $\c$ does not necessarily map $\V^i$ to $\VV^i$ for $i>1$.
We denote the restriction of $\c$ to $\V^i$ by $\c^i$ and define the blocks
 by  $\c^{j,i}=\PP^{j} \circ \c^i : \V^i \to \VV^j$.
 Since $\mathcal{V}^i_x=\c(x) \vv^i_x$, we have $\c^i (\V^i) \subset  \mathcal{V}^i$ and thus $\c^{j,i}=0$ for $j>i$,
so that $\c$ also has the block triangular structure.

First we show that the diagonal blocks $\c^{i,i} : \V^i \to \VV^j$ are su-$\beta$-H\"older, for $i=1,\dots,k$. 
For this we note that $\c^{i,i}$ gives a measurable conjugacy between su-$\beta$-H\"older cocycles 
$\A^{(i)}$  on $ \V^i$ and $\B^{(i)}$  on $ \VV^j$.  Recall that $\A^{(i)}$ is conformal with respect 
to metric $\sigma_i$.  The cocycle $\B^{(i)}$ has holonomies as in Proposition \ref{existence of holonomies} induced, via the quotient, by the holonomies of the cocycle $\B$, which is su-$\beta$-H\"older and fiber bunched  by the assumption.
%or one can see that $\B^{(i)}$ is also fiber bunched? 
Now part (ii) of Theorem \ref{measurable cohomology}  shows that $\c^{i,i}$ is su-$\beta$-H\"older. Also, pushing the metric $\sigma_i$ by $\c^{i,i}$ to $ \VV^j$ we obtain a 
Riemannian metric $ \tau_i$ for which $\B^{(i)}$ is conformal and $\psi \B^{(i)}$ is isometric.

%We conclude that both $\A^{(i)}$ and $\B^{(i)}$ are conformal and have the same norm $\psi^{-1}$ with respect to metrics  $\tilde \sigma_i$ and  $\tilde \tau_i$ respectively. This will allow us to use Theorems \ref{QC} and \ref{twist-meas}.

We will now show inductively that the restriction of $\c$ to $\vv^i$ is su-$\beta$-H\"older for $i=1,\dots,k$.
The base case $i=1$ follows from the previous paragraph since $\c|_{\vv^1}=\c^{1,1}$.

Now we describe the inductive step. Assuming that the restriction of $\c$ to $\vv^{i-1}$ is 
su-$\beta$-H\"older we show that so is the restriction to $\vv^{i}$. Since $\vv^i=\V^i \oplus \vv^{i-1}$,
it suffices to show that the  restriction $\c^{i}$ of $\c$ to $\V^{i}$ is  also su-$\beta$-H\"older.
We establish this for its components $\c^{j,i}$, $j=i,\dots ,1$, by induction. In the base case
$j=i$ we already know that the diagonal block $\c^{i,i}$ is su-$\beta$-H\"older. 

Now we show that $\c^{i-\ell,i}$, with $0<\ell<i$, is su-$\beta$ H\"older assuming that $\c^{i-j,i}$  is su-$\beta$-H\"older for $j=0,1,\dots \ell-1$.
Using the conjugacy equation
$$
\B_x \circ \c_x =\c_{fx}  \circ \A_x
$$
and equating the $(i-\ell,i)$ components we obtain
$$
\begin{aligned}
&\B^{i-\ell,i-\ell}_x \circ \c^{i-\ell,i}_x + \B^{i-\ell,i-\ell+1}_x \circ \c^{i-\ell+1,i}_x + \dots +\B^{i-\ell,i}_x \circ \c^{i,i}_x \\
&
=\,\c^{i-\ell,i-\ell}_{fx}  \circ \A^{i-\ell,i}_x + \c^{i-\ell,i-\ell+1}_{fx}  \circ \A^{i-\ell+1,i}_x + \dots +\c^{i-\ell,i}_{fx}  \circ \A^{i,i}_x
\end{aligned}
$$
and hence
\begin{equation} \label{CD}
\c^{i-\ell,i}_x  = (\B^{i-\ell,i-\ell}_x)^{-1} \circ \c^{i-\ell,i}_{fx}  \circ \A^{i,i}_x \, + D_x
\end{equation}
where
$$
\begin{aligned}
D_x  =\,\, & (\B^{i-\ell,i-\ell}_x)^{-1} \circ(\c^{i-\ell,i-\ell}_{fx}  \circ \A^{i-\ell,i}_x + \dots +\c^{i-\ell,i-1}_{fx} \circ \A^{i-1,i}_x ) \\
&-(\B^{i-\ell,i-\ell}_x)^{-1} \circ ( \B^{i-\ell,i-\ell+1}_x \circ \c^{i-\ell+1,i}_x + \dots +\B^{i-\ell,i}_x \circ \c^{i,i}_x ).
\end{aligned}
$$
Then equation \eqref{CD} is of the form \eqref{twisteq} with
$$
\p _x =D_x,\;\; \eta_x=\c^{i-\ell,i}_x, \;\;\text{and}\;\;  \F_x(\q_{fx})=  (\B^{i-\ell,i-\ell}_x)^{-1} \circ \eta_{fx}  \circ \A^{i,i}_x, % \;\; \E=L(\V^i,\VV^{i-\ell}),
$$
where we view $\c^{i-\ell,i}_x$ and $D_x$ as sections of the bundle $\L(\V^i,\VV^{i-\ell})$ 
whose fiber at $x$ is the space $L(\V^i_x,\VV^{i-\ell}_x)$ of linear maps from $\V^i_x$ to 
$\VV^{i-\ell}_x$. This is a subbundle of the  $\beta$-H\"older  bundle $\L=\L(\E,\E')$, 
where we view $L(\V^i_x,\VV^{i-\ell}_x)$ as the subspace of those operators in $L(\E_x,\E_x')$ 
for which all other blocks, with respect to the splittings $\E_x=\, \oplus \V^i_x$ and 
$\E_x'=\oplus \, \VV^i_x$,   are zeros.
Since the splittings are su-$\beta$-H\"older, so is the subbundle
$\L(\V^i,\VV^{i-\ell})$. We also have that $D_x$ is su-$\beta$-H\"older since we inductively 
know that all its terms are su-$\beta$-H\"older. Indeed, for the second term this follows 
from the assumption that $\c^{i-j,i}$  is su-$\beta$-H\"older for $j=0,1,\dots \ell-1$,
and for the first term this follows from the assumption that the restriction of $\c$ to 
 $\vv^{i-1}$ is su-$\beta$-H\"older and hence so are all blocks $\c^{i-\ell,m}$ with $m\le i-1$.
We view $\F$ as a linear cocycle on the bundle 
$\L(\V^i,\VV^{i-\ell})$ over $f^{-1}$, and it is su-$\beta$-H\"older since so are $\B^{i-\ell,i-\ell}$ 
and $ \A^{i,i}$. 
Moreover, $\F$ is uniformly bounded since the cocycles $\psi \B^{i-\ell,i-\ell}$ and $\psi \A^{i,i}$ 
are isometric  respect to $ \sigma_i$ and  $ \tau_i$
and hence 
$$\|\F_x(\q_{fx})\| \le \|(\psi \B^{i-\ell,i-\ell}_x)^{-1}\| \cdot \|\q_{fx}\| \cdot \|\psi \A^{i,i}_x\| = \|\q_{fx}\|.$$
Thus we can apply Theorem \ref{twist-meas} and conclude that $\c^{i-\ell,i}$ is su-$\beta$-H\"older.
\vskip.1cm

The argument above applies to $\ell=1, \dots i-1$ and we conclude that all 
$\c^{1,i}, \dots  ,\c^{i,i}$ are su-$\beta$-H\"older.  We also recall that  $\c^{j,i}=0$ for $j>i$,
and thus the  restriction $\c^{i}$ of $\c$ to $\V^{i}$ is  also su-$\beta$-H\"older.
This proves that so is the restriction of $\c$ to $\vv^{i}$ and completes the inductive step.
We conclude that $\c$ is su-$\beta$-H\"older, completing the proof of Theorem~\ref{main}.

%%%%%%%%%%  Proof of Theorem Constant Cocycle %%%%%%%%%%%%
%%%%%%%%%%%%%%%%%%%%%%%%%%%%%%%%%

 \section{Proofs of Theorems \ref{constant 1exp} and \ref{constant perturb}} \label{proof of constant}

\subsection {Proof of Theorem \ref{constant 1exp}} 
Let $A$ be the matrix generating cocycle $\A$ in Theorem \ref{constant 1exp}. 
The one exponent assumption  means that all eigenvalues of $A$ have the same modulus 
$\rho$.  Then the real Jordan canonical form of matrix $\rho^{-1}A$ has block triangular structure with 
orthogonal bocks on the diagonal. This yields the corresponding flag of invariant subbundles for 
the cocycle $\A$ with properties as in Theorem \ref{main}. Hence Theorem \ref{main} implies Theorem \ref{constant 1exp}.

\subsection{Proof of Theorem \ref{constant perturb}} 
Now we deduce Theorem  \ref{constant perturb} from Theorem \ref{constant 1exp}. Let $A$ be the matrix generating the cocycle 
$\A$ and let $\rho_1 < \dots <\rho_\ell$ be the distinct moduli of its eigenvalues. We consider
the corresponding invariant splitting
\begin{equation} \label{splitA}
 \R^d = E^1 \oplus  \dots \oplus E^\ell,
\end{equation}
where $E^i$ denotes the sum of the generalized eigenspaces of $A$ corresponding to the eigenvalues
of modulus $\rho_i$. This gives a splitting of the trivial bundle $\E=\M\times \R^d$
into $\A$-invariant constant subbundles $E^i$. For any $\e>0$ there is a suitable norm on $\R^d$ 
with resect to which we have
\begin{equation} \label{rateA}
 (\rho_i-\e)^n \le \| \A^n u \| \le  (\rho_i+\e)^n
\quad \text{for any unit vector }u\in E^i \text{ and }n\in \Z.
\end{equation}

Let  $B(x)=\B_x:\M\to GL(d,\R)$ be the generator of the cocycle $\B$. If $B$ is
sufficiently $C^0$ close to $A$, then $\E=\M\times \R^d$ has a continuous $\B$-invariant
splitting $C^0$ close to \eqref{splitA},
$$
\R^d = \EE^1_x \oplus \dots \oplus \EE^\ell_x,
$$
for which estimates similar to  \eqref{rateA} hold,
%\begin{equation} \label{rateB}
$$
 (\rho_i-2\e)^n \le \| \B^n u \| \le  (\rho_i+2\e)^n
\quad \text{for any unit vector }u\in \EE^i  \text{ and }n\in \Z.
$$
%\end{equation}
Moreover, for a H\"older $\B$ it is well known that the splitting is also H\"older with some 
exponent $\beta>0$, which may be smaller than that of $\B$. 
See for example \cite[Lemma 5.1]{KSW}, which gives estimates for  $\beta$ in terms of $\rho_i$ and $f$.
We conclude that all restrictions $\B_i=\B|_{\E^i}$ are $\beta$-H\"older and hence are fiber bunched
 if $\e$ is sufficiently small.

Let $\c$ be a measurable conjugacy between $\A$ and $\B$.
We claim that  $\c$ maps $E^i$ to $\EE^i$, that is,
$\c_x (E^i)=\EE^i_x$ for $\mu$ a.e.~$x$. Indeed, by Lemma \ref{equal exp}, for $\mu$ a.e.~$x$ and for each unit vector $u \in E^i_x$
the forward and backward Lyapunov exponent of $\c_x(u)$ is $\ln \rho_i$.
This yields that $\c_x(u) \in \EE^i$, as having a non-zero component in another $\EE^j$ would
imply having forward or backward Lyapunov exponent under $\B$ different from $\ln \rho_i$
if $\e$ is sufficiently small.
Then $\c_i=\c|_{ E^i}$ is a measurable conjugacy between the constant cocycle $\A_i=\A|_{E^i}$ with 
one Lyapunov exponent and the $\beta$-H\"older fiber bunched cocycle $\B_i$.
By Theorem \ref{constant 1exp} each $\c_i$,  $i=1,\dots,\ell$,
is su-$\beta$-H\"older  and hence so is $\c$.
 This completes the proof of Theorem \ref{constant perturb}.

%%%%%%%%%%%%%%%%%%%%%%%%%%%%%%%%%
%%%%%%%%               bibliography                %%%%%%%%%

\end{document}